\documentclass[twoside,12pt,reqno]{amsart}

\date{\today}

\usepackage{accents}
\usepackage[latin1]{inputenc}
\usepackage{pstricks}
\usepackage{enumerate}	
\usepackage{graphicx,color}
\usepackage[english]{babel}
\usepackage[toc,page]{appendix}
\usepackage{amssymb, mathrsfs, amsfonts, amsmath,amsthm}
\usepackage{mathtools}
\usepackage{amsbsy, hyperref}
\usepackage{latexsym}
\usepackage{booktabs}
\usepackage{tikz}
\usepackage{eurosym} 
\usepackage{newfloat}
\usepackage{hyperref}
\usepackage[alphabetic, msc-links]{amsrefs}

\definecolor{r}{rgb}{.9,0.1,.3}

\numberwithin{equation}{section}

\newtheorem*{question}{Question}
\newtheorem{Theorem}{Theorem}[section]
\newtheorem{Proposition}[Theorem]{Proposition}
\newtheorem{Definition}[Theorem]{Definition}

\newtheorem{Remark}[Theorem]{Remark}
\newtheorem{Lemma}[Theorem]{Lemma}

\newcommand{\R}{\mathbb{R}}

\newcommand{\Div}{\operatorname{div}}
\newcommand{\Sing}{\operatorname{Sing}}
\newcommand{\graph}{\operatorname{Graph}}
\newcommand{\Area}{\operatorname{Area}}

\author[E. S. Gama]{Eddygledson S. Gama}
\address[Gama]{
  Departamento de Matem\'atica,
  Universidade Federal do Cear\'a, Bloco 914, Campus do Pici,
  Fortaleza, Cear\'a, 60455-760, Brazil.
  Departamento de Geometr\'\i{}a y Topolog\'\i{}a,
  Universidad de Granada,
  18071 Granada, Spain.
}
\email{eddygledson@gmail.com}

\author[E. Heinonen]{Esko Heinonen}
\address[Heinonen]{
Departamento de Geometr\'i{}a y Topolog\'\i{}a,
  Universidad de Granada,
  18071 Granada, Spain.
Department of Mathematics and Statistics,
University of Helsinki, Finland.
}
\email{ea.heinonen@gmail.com}

\author[J. H. Lira]{Jorge H. De Lira}
\address[Lira]{
  Departamento de Matem\'atica,
  Universidade Federal do Cear\'a, Bloco 914, Campus do Pici,
  Fortaleza, Cear\'a, 60455-760, Brazil.
}

\email{jorge.lira@mat.ufc.br}
\author[F. Mart\'in]{Francisco Mart\'\i{}n}
\address[Mart\'in]{
  Departamento de Geometr\'i{}a y Topolog\'\i{}a,
  Universidad de Granada,
  18071 Granada, Spain.
}
\email{fmartin@ugr.es}

\thanks{
E. S. Gama was supported by Capes/PDSE/88881.132464/2016-01. E. Heinonen was supported by the Finnish Academy of Science and Letters and by the Centre of Excellence in Analysis and Dynamics Research (Academy of Finland, decision 271983). J. H. de Lira was supported by PRONEX/FUNCAP/CNPq PR2-0101-00089.01.00-15. and
CNPq/Edital Universal 409689/2016-5. F. Mart\'in was partially supported by the MINECO/FEDER grant MTM2017-89677-P and  by the
Leverhulme Trust grant IN-2016-019.}

\title[Jenkins-Serrin graphs in $M\times\R$]{Jenkins-Serrin graphs in $M\times\R$}

\begin{document}
\maketitle
\begin{abstract} The so called Jenkins-Serrin problem is a kind of Dirichlet problem for graphs with
prescribed mean curvature that combines, at the same time, continuous boundary data with regions of the boundary where
the boundary values explodes either to $+\infty$ or to $-\infty.$ We give a survey on the development of
Jenkins-Serrin type problems over domains on Riemannian 
manifolds. The existence of this type of graphs imposes restrictions  on the geometry of the boundary of these domains.
We also improve some earlier results by proving Theorem \ref{mainth}, and prove the existence of translating Jenkins-Serrin graphs (Theorem \ref{Existence type 1}).\end{abstract}

\section{Introduction}
Let $(M,\sigma)$ be a Riemannian $n$-manifold (not necessarily complete) and let $\Omega \subset M$ be a domain whose boundary is piecewise smooth. The graph $\Sigma=\graph(u)$ of a smooth function $u\colon\Omega \rightarrow \R$ is a smooth hypersurface in the product manifold $M\times \R$ endowed with the product metric $g_0\coloneqq \sigma + {\rm d}t^2$. Moreover, it is standard to check that if $H\colon M\to\R$ is a smooth function, then $\Sigma$ has prescribed mean curvature $H$ if and only if

\begin{equation}\label{PDE}
\Div\left(\frac{\nabla u}{\sqrt[]{1+|\nabla u|^2}}\right)=H,
\end{equation}
for all $x\in\Omega$, where $\nabla$ and $\Div$ are the gradient and divergence with respect to $\sigma.$ 

In this paper we will work in the cases when either $\Sigma$ is minimal, or $\Sigma$ has constant mean curvature, or $\Sigma$ is a translating soliton of the mean curvature flow.  Moreover, we will adopt the convention that {\em the mean curvature is just the trace of the second fundamental form.}

First of all, let us explain what we mean by a Jenkins-Serrin solution of \eqref{PDE}. 
Given any Riemannian manifold $M$ of dimension $n \geq 2$, let $\Omega\subset M$ be a bounded Lipschitz domain 
whose boundary can be written as 
$$\partial \Omega = \Gamma_0 \cup \Gamma_1 \cup \Gamma_2 \cup N$$
where $ \Gamma_0$, $\Gamma_1$, $\Gamma_2$ are open disjoint subsets of class $C^1$ and $N$ is a closed subset 
such that  the $(n-1)$-dimensional Hausdorff measure $\mathcal{H}_{n-1}(N)=0.$
 The Jenkins-Serrin solution of \eqref{PDE} with continuous data $f_0\colon\Gamma_0\to \R$ is a smooth function 
 $u\colon\Omega\to\R$ that satisfies \eqref{PDE}, \(u|_{\Gamma_0}=f_0\) and 
\begin{enumerate}
\item[i.]$u\to+\infty$ as $x\to \Gamma_1$;
\item[ii.]$u\to-\infty$ as $x\to \Gamma_2$.
\end{enumerate}
As we will see later, the conditions i and ii impose important restrictions on the extrinsic geometry of $\Gamma_1$
and $\Gamma_2$, in general.

Probably, the best known example of Jenkins-Serrin solutions is given by the classical Scherk's minimal surface, that is the graph of the function $f(x,y)=\ln (\cos(x) /\cos(y))$ over the domain $\Omega=(-\pi/2,\pi/2)\times (-\pi/2,\pi/2)$.  This idea of construction was generalized by H. Jenkins and J. Serrin in  \cite{Jenkins-Serrin}, but 
 before stating their theorem, we need some notation. Let $\Omega\subset \R^2$ be a bounded domain whose boundary consists of a finite number of open straight arcs $A_1,\ldots,A_r$ (the set $\Gamma_1$), $B_1\ldots,B_s$ (the set $\Gamma_2$), and a finite number of convex arcs $C_1,\ldots,C_t$ (the set $\Gamma_0$) with respect to $\Omega$ and the end points of these arcs (the set $N$).  Let $\mathcal{P}$ be a polygon whose vertices are chosen among the vertices of $\Omega;$ this polygon will be called {\em an admissible polygon}. We are going to use also the following notation:
\[\alpha(\mathcal{P})=\sum_{A_i\subset\mathcal{P}}||A_i||,\ \beta(\mathcal{P})=\sum_{B_i\subset\mathcal{P}}||B_i|| \ \text{ and } \ \ell(\mathcal{P})={\rm perimeter}(\mathcal{P}),\]
where $||A||$ denotes the arc length of $A$.


\begin{Theorem}[Jenkins-Serrin]\label{Jenkins-Serrin 1}
Let $\Omega\subset \R^2$ be as above. 
Assume also that no two arcs $A_i$ and no two arcs $B_i$ have a common endpoint. Given any continuous data $f_i\colon C_i\to\R$, then there exists a Jenkins-Serrin solution $u\colon \Omega\to\R$ for the minimal graph equation with continuous data \(u|_{C_i}=f_i\) if, and only if, for any admissible polygon $\mathcal{P}$ we have
\begin{equation}\label{structure}
2\alpha(\mathcal{P})<\ell(\mathcal{P})\ \operatorname{and}\ 2\beta(\mathcal{P})<\ell(\mathcal{P}).
\end{equation}
If $\{C_i\}=\varnothing$, we also require that $\alpha(\partial \Omega)=\beta(\partial\Omega)$ for $\mathcal{P}=\partial\Omega.$ Moreover, if $\{C_i\}\neq\varnothing,$ then $u$ is unique, and if $\{C_i\}=\varnothing,$ then $u$ is unique up to adding a constant.
\end{Theorem}

In \cite{Jenkins-Serrin}, Jenkins and Serrin proved also that the existence of minimal Jenkins-Serrin solutions imposes several restrictions over the domain.

\begin{Theorem}[Jenkins-Serrin]\label{Jenkins-Serrin 3}
Let $\Omega\subset \R^2$ be a bounded domain. Suppose that $u\colon \Omega\to\R$ is  a solution of the minimal graph equation 
satisfying $u\to\pm\infty$ as $x\to\gamma,$ where $\gamma$ is a smooth connected component of $\partial \Omega.$ Then $\gamma$ is a geodesic.
\end{Theorem}

J. Spruck \cite{Spruck-Infinite} extended Jenkins' and Serrin's work for the CMC case. He proved the existence of CMC Jenkins-Serrin solutions for very specific domains in $\R^2$. The setting in \cite{Spruck-Infinite} was the following:

Let $\Omega\subset \R^2$ be a bounded domain whose boundary consists of finite number of open arcs $A_1,\ldots,A_r,B_1\ldots,B_s$, $C_1,\ldots,C_t$ and the endpoints of these arcs. Suppose that
the curvatures of these arcs satisfy 
	\begin{equation}\label{spruck-domain}
	\kappa(A_i)=H  , \quad \kappa(B_i)=-H, \quad \kappa(C_i)\geq H.
	\end{equation}

Now, we say that $\mathcal{P}$ is  {\em admissible} if $\mathcal{P}$ is a simple closed curvilinear polygon whose sides
consists of circular arcs of curvature $\kappa=\pm H$ and whose vertices belong to the set of end points
of the arcs $\{A_i\}$ and $\{B_i\}$. As before, $\alpha$ and $\beta$, respectively, denotes the total length of the arcs $A_i$ and $B_i$, and $\ell$ denotes the perimeter of $\mathcal{P}$.

\begin{Theorem}[Spruck]\label{Spruck 0}
Let $\Omega\subset \R^2$ be a domain as above satisfying \eqref{spruck-domain} with 
$\{C_i\}\neq\varnothing$ and $|B_i|< \pi /H$.
Assume also that no two arcs $A_i$ and no two arcs $B_i$ have a common endpoint and that the domain 
$\Omega^*$ formed by reflection\footnote{
We reflect $B_i$ with respect to the line passing through the end points. In this way we obtain an arc
$B_i^*$ with curvature $\kappa(B_i^*)= H.$}  of the arcs \(B_i\) is contained in a disk of radius $1/H$. Given any continuous data $f_i \colon C_i\to\R$, there exists a unique Jenkins-Serrin solution $u \colon \Omega\to\R$ with mean curvature $H$ and satisfying \(u|_{C_i}=f_i\), if and only if, for any admissible polygon $\mathcal{P}$ we have
\begin{equation*}\label{structure CMC}
2\alpha(\mathcal{P})<\ell(\mathcal{P})+H\operatorname{Area}(\mathcal{P})\ \operatorname{and}\ 2\beta(\mathcal{P})<\ell(\mathcal{P})-H\operatorname{Area}(\mathcal{P}).
\end{equation*}
\end{Theorem}

In the case $\{C_i\}=\varnothing,$ he obtained that the necessary and sufficient conditions for the existence
of the solution are

\begin{equation*}\label{structure CMC*}
2\alpha(\mathcal{P})<\ell(\mathcal{P})+H\operatorname{Area}(\mathcal{P})\ \operatorname{and}\ 2\beta(\mathcal{P})<\ell(\mathcal{P})-H\operatorname{Area}(\mathcal{P})
\end{equation*}
and 
\begin{equation*}\label{structure CMC**}
\alpha(\mathcal{P})=\beta(\mathcal{P})+H\operatorname{Area}(\Omega).
\end{equation*}

In higher dimensions the question is more difficult but
Spruck proved the following local existence of CMC Jenkins-Serrin solution.
\begin{Theorem}[Spruck]\label{Spruck 2}
Let $S$ be a $C^2$ hypersurface in $\R^n$ with mean curvature $H(>0).$ Let $p$ be an interior point of $S$, and $B_\varepsilon(p)$ a ball with center $p$ and radius $\varepsilon$. 
For small $\varepsilon,$ let $D^+$ and $D^-$ be the domains bounded by $S$ and $\partial B_{\varepsilon}(p)$ so that the mean curvature $H_S= H$ with respect to $D^+$ and $H_S=-H$ with respect to $D^-$.
Then for $\varepsilon$ small enough, there exists a Jenkins-Serrin solution $u$ in $D^+$ (respectively $D^-$) with boundary values $+\infty$ (respectively $-\infty$) on $S$
and prescribed continuous data $f$ on the remainder of the boundary.
\end{Theorem}

Moreover, he obtained an analogous result to Theorem \ref{Jenkins-Serrin 3} making a clever use of suitable barriers. 

\begin{Theorem}[Spruck]\label{Spruck 3}
Let $\Omega\subset \R^n$ be a domain that is bounded in part by a $C^2$ hypersurface $S$.
Suppose that $u \colon \Omega\to\R$ is a solution of the CMC equation with constant mean curvature $H(>0)$ satisfying $u\to\infty$ $($respectively $u\to-\infty$ $)$ as $x\to S$. 
Then $H_S=H$ (respectively $-H$), where $H_S$ denotes the mean curvature of $S$ in $\R^n.$
\end{Theorem}
Further results about Jenkins-Serrin solutions in higher dimensions can be found in \cite{Giusti, Massari}.

B. Nelli and H. Rosenberg \cite{Nelli-Rosenberg} extended the Jenkins-Serrin theory for domains in the hyperbolic plane $\mathbb{H}^2$, and  after that, A. L. Pinheiro \cite{Pinheiro} to more general Riemannian surfaces.
%


Although we are only considering the case of bounded domains, we would like to mention that  L. Mazet, M. Rodr\'iguez and H. Rosenberg \cite{Mazet-Rodriguez-Rosenberg} proved that, in the minimal case, it is possible to get Jenkins-Serrin solutions over non-compact domains in the hyperbolic plane $\mathbb{H}^2.$ Their result generalizes previous results obtained by P. Collin and H. Rosenberg in \cite{Collin-Rosenberg}. To state the result, we use the following notation.

We say that a domain $\Omega\subset\mathbb{H}^2$ is an ideal Scherk domain if $\partial_\infty\Omega$ consists of a finite number of geodesic arcs $A_i,B_i,$ a finite number of convex arcs $C_i$ (with respect to $\Omega$), a finite number of open arcs $D_i$ at $\partial_\infty\mathbb{H}^2,$ together with their endpoints (called the vertices of $\Omega$) and no two $A_i$ and no two $B_i$ have a common end point. Now let $\Omega$ be an ideal Scherk domain; a polygon $\mathcal{P}\subset\Omega$ is said to be an admissible polygon if its vertices are among the vertices of $\Omega.$ For any ideal vertex $p_i$ of $\Omega$ at $\partial_\infty\mathbb{H}^2$, let $H_i$ be a horocycle at $p_i$ such that for any $i\neq j$ we have $H_i\cap H_j=\varnothing$, and $H_i$ does not intersect the bounded edges of $\partial\Omega.$ Finally, given any admissible polygon $\mathcal{P}\subset\Omega$, denote by $\Gamma(\mathcal{P})$ the part of $\partial\mathcal{P}$ outside the horocycles $H_i$ and call
\[
\alpha(\mathcal{P})=\sum_{A_i\subset\mathcal{P}}||A_i\cap\Gamma(\mathcal{P})||\ \text{and} \ \beta(\mathcal{P})=\sum_{B_i\subset\mathcal{P}}||B_i\cap\Gamma(\mathcal{P})|| 
\]

With this notation, we can state their theorem. 

\begin{Theorem}[Mazet-Rodr\'iguez-Rosenberg]\label{Mazet-Rodriguez-Rosenberg}
Suppose that there is at least one edge $C_i$ or $D_i$ in $\partial_\infty\Omega$. Then there exists a Jenkins-Serrin solution with continuous data $f_i \colon C_i\to\R$ and $g_i \colon D_i\to\R$ if, and only if, the horocycles $H_i$ can be chosen so that 
\begin{equation*}\label{structure for H^2}
2\alpha(\mathcal{P})< ||\Gamma(\mathcal{P})||\ \operatorname{and}\ 2\beta(\mathcal{P}) < ||\Gamma(\mathcal{P})||.
\end{equation*}
for any admissible polygon $\mathcal{P}$ in $\Omega$.
\end{Theorem}

Around the same time, J. G\'alvez and H. Rosenberg \cite{Galvez-Rosenberg} proved similar result for simply connected and complete Riemannian surfaces whose Gaussian curvature is bounded from above by a negative constant.


Concerning the CMC case, L. Hauswirth, H. Rosenberg and J. Spruck \cite{Hauswirth-Rosenberg-Spruck} proved the existence of Jenkins-Serrin solutions in $\mathbb{H}^2$ and $\mathbb{S}^2$. Later A. Folha and H. Rosenberg \cite{Folha-Rosenberg} generalized these results to the case of Hadamard surfaces.

We would like to point out that Folha and Rosenberg also proved the following theorem that provides a relationship between the mean curvature of the graph and the mean curvature of the boundary of the domain.

\begin{Theorem}[Folha-Rosenberg]\label{Folha-Rosenberg 2}
Let $M$ be a Hadamard surface and $\Omega\subset M$ be a bounded domain. Suppose that $u \colon \Omega\to\R$ is a solution of CMC equation with constant mean curvature $H(>0)$ satisfying $u\to\infty$ $($respectively $u\to-\infty$ $)$ as $x\to\gamma,$ where $\gamma$ is a smooth connected component of $\partial \Omega.$ Then $H_\gamma=H$ (respectively $-H$), where $H_{\gamma}$ denotes the mean curvature of $\gamma$ in $M.$
\end{Theorem}

Notice that a fact repeatedly appears  in Theorems \ref{Jenkins-Serrin 3},  \ref{Spruck 3} and  \ref{Folha-Rosenberg 2}. 
Namely, they showed that if we have a Jenkins-Serrin solution over a domain $\Omega$, then the mean curvature of $\partial\Omega$ and the mean curvature of the surface given by the graph of the Jenkins-Serrin solution are equal, up to a sign. This suggests the following question:

\begin{question} \em How does the existence of complete graphs with prescribed mean curvature over a regular domain affect the geometrical structure of its boundary? \end{question}

In relation to the translating graphs, L. Shahriyari \cite{Shariyari} obtained a complete result for translating graphs over regular domains in $\R^2$. The main idea in Shahriyari's proofs is to use Schoen's type estimates for stable minimal surfaces in $3$-dimensional manifolds, an important tool in Colding \& Minicozzi theory, to get information about the behaviour of the graph as we approach a component $\gamma$ of the boundary, like Folha and Rosenberg do in \cite{Folha-Rosenberg}. Finally, she concluded that $\gamma$ is a geodesic or has constant principal curvature in a CMC case. The main drawback of their method is that their proofs do not work in higher dimension. 

Using a quite different strategy M. Eichmair and J. Metzger \cite{eichmair-metzger} were able to improve earlier result by Spruck (Theorem \ref{Spruck 0}), in the case where $\{C_i\}=\emptyset$, into Riemannian manifolds of dimension $2\le n\le 7$. Namely, they didn't need the existence of an auxiliary domain $\Omega^*$ nor a priori assumptions about existence of subsolutions, see \cite[Section 6.1]{eichmair-metzger}. On the other hand, inspired by a classical compactness results of stable hypersurfaces by R. Schoen and L. Simon \cite{Schoen-Simon}, Eichmair and Metzger gave a very simple proof that makes it possible to get an extension of Theorem \ref{Jenkins-Serrin 3}, Theorem \ref{Spruck 3} and Theorem \ref{Folha-Rosenberg 2} for any $n$-dimensional Riemannian manifold for $2\leq n\leq7$. We will show that the same method can be applied for any Riemannian manifold without any restriction on the dimension. More precisely, we prove the following theorem.

\begin{Theorem} \label{mainth}
Let $M$ be a complete Riemannian manifold and $\Omega\subset M$ be a domain with piecewise smooth boundary. Let $\Lambda\subset\partial\Omega$ be a smooth connected region and let $\Sigma$ be the graph of a smooth function $u$ on $\Omega$ in $M\times\R$.
Suppose $\Sigma$ is complete as we approach $\Lambda$. Then, we have
\begin{enumerate}[a.]
\item If $\Sigma$ is a translating soliton, then $H_{\Lambda}
=0$;
\item If $\Sigma$ has constant mean curvature $H_0$ (with respect to the upward pointing normal vector field), then $H_{\Lambda}= H_0$, up to the sign.
\end{enumerate}
\end{Theorem}

The crucial ingredient in our proof is a local area estimate for the components of the graph in closed balls around the cylinder over the boundary. These area estimates, combined with a suitable compactness theorems (Theorem \ref{Wickramasekera} in the minimal case
and Theorem \ref{Bellettini-Wickramasekera} for CMC hypersurfaces), yield precise information of the geometric structure of the domain through a surprisingly simple local analysis.

In the last section we consider the existence of Jenkins-Serrin type solutions to the translating soliton equation and prove the following result.
\begin{Theorem}\label{Existence type 1}
Let $\Omega\subset M$ be an admissible domain with $\{B_i\}=\varnothing.$ Given any continuous data $f_i \colon C_i\to\R$, there exists a Jenkins-Serrin solution $u\colon \Omega\to\R$ for translating equation with continuous data \(u|_{C_i}=f_i\) if for any admissible polygon $\mathcal{P}$ we have
\begin{equation*}
2\alpha(\mathcal{P})<\ell(\mathcal{P}).
\end{equation*}
\end{Theorem}



{\bf Acknowledgements.} We would like to thank Brian White for his valuable comments, in particular concerning Theorem \ref{eucl-thm}. E. S. Gama is very grateful to the Institute of Mathematics at the University of Granada for its hospitality during the time the research and preparation of this article were conducted.

\section{Compactness theorem for stable integral varifolds}\label{back}

In this section, we give a brief review of the results that we need about varifolds. We refer the reader to \cite{SIMON} for an elegant introduction on the subject. Before we state the compactness theorem for varifolds, recall the following concepts of weak and strong convergence.  

\begin{Definition}\label{Weakly Convergence}
Let $\left\{S_{i}\right\}$ be a sequence of varifolds in a smooth manifold $\bar M$. We say that $\{S_{i}\}$ converges weakly to a varifold $S$, if for every smooth function $\varphi \colon \bar M\rightarrow\mathbb{R}$ with compact support we have
\begin{equation*}
\lim_{i\rightarrow \infty}\int_{\bar M}\varphi\, {\rm d} \mu_{S_{i}}=\int_{\bar M}\varphi\, {\rm d}\mu_{S},
\end{equation*}
where $\mu_{S_{i}}$ is the Radon measure in $\bar M$ associated to varifold $S_{i}$. If the sequence $\left\{S_{i}\right\}$  converges weakly to $S$, then we will write $S_{i}\rightharpoonup S$.
\end{Definition}

Let $S$ be a hypersurface in a Riemannian manifold $\bar M$. Given $p\in S$ and $r>0$ we denote by $B_p(r)$ the tangent ball in $T_p S\subset T_p \bar M$ centered at $0\in T_p S$ with radius $r$.
\begin{equation*}
B_{r}(p)=\left\{v\in T_{p}S\colon |v|<r\right\}
\end{equation*}
Let $N$ be a unit normal vector to $S$ at $p$. 
Fixed a sufficiently small $\varepsilon>0$, let $W_{r,\varepsilon}(p)$  be the solid cylinder around $p$, that is,
\begin{equation*}
W_{r,\varepsilon}(p) \coloneqq \left\{\exp_p(q+tN) \colon q\in B_{r}(p)\ \operatorname{and}\ |t|<\varepsilon\right\},
\end{equation*}
where  $\exp$ is the exponential map of the manifold $\bar M$.  Given a smooth function $f:B_{r}(p)\rightarrow (-\varepsilon, \varepsilon)$, the set
\begin{equation*}
{\rm Graph}[f] \coloneqq\left\{\exp_p(q+u(q)N) \colon q\in B_{r}(p)\right\},
\end{equation*}
is called the graph of $f$ over $ B_{r}(p).$ Finally, the convergence in $C^{\infty}$-topology can be defined as follows.
\begin{Definition}\label{Smooth Convergence}
Let $\left\{S_{i}\right\}$ be a sequence of smooth codimension one submanifolds of a smooth manifold $\bar M$. We say that $\{S_{i}\}$ converges in $C^{\infty}$-topology with finite multiplicity to a smooth embedded submanifold $S$ if 
\begin{itemize}
\item[a.] $S$ consists of accumulation points of $\{S_{i}\}$, that is, for each $p\in S$ there exists a sequence $\{p_{i}\}$ such that $p_{i}\in S_{i}$, for each $i\in \mathbb{N}$, and $p=\displaystyle\lim_{i}p_{i}$;

\item[b.] For every $p\in S$ there exist $r,\varepsilon>0$ such that $S\cap W_{r,\varepsilon}(p)$ can be represented as a graph of a function $f$ over $B_{r}(p)$;

\item[c.] For $i$ large enough, the set $S_{i}\cap W_{r,\varepsilon}(p)$ consists of a finite number, say $k$, independent of $i$, of graphs of functions $f_{i}^1,\ldots,f^{k}_{i}$ over $B_{r}(p)$ that converge smoothly to $f.$
\end{itemize}
The  multiplicity of a given point $p\in S$ is defined by $k$. If $\left\{S_{i}\right\}$ converges smoothly to $S$, then we will write $S_{i}\rightarrow S$.
\end{Definition}

In this setting, the strong compactness theorem for integral varifolds can be stated as follows. 
\begin{Theorem}[Strong compactness theorem]\label{Wickramasekera} 
Let $\left\{S_{i}\right\}$ be a sequence of stable minimal hypersurfaces in Riemannian manifold $\bar M^{n+1}$, without boundary, such that 
\begin{equation*}
\limsup\mu_{S_i}(\bar M)<\infty,
\end{equation*}
where $\mu_{S_i}$ is the (Riemannian) measure associated to $S_i$. Suppose that there exists a sequence $\{p_i\}$ so that $p_i\in S_i$ and $p_i\to p_{\infty}\in \bar M.$ Then there exist a closed set, $\operatorname{sing}S_{\infty}$, and a stationary integral varifold, $S_{\infty}$, such that a subsequence of $\left\{S_{i}\right\}$ converges smoothly to 
$S_{\infty}$ away from $\operatorname{sing}S_{\infty}$ and $S_i\rightharpoonup S_{\infty}$ with $p_\infty\in \operatorname{spt} S_\infty$. Moreover, the set $\operatorname{sing}S_{\infty}$ has Hausdorff dimension at most $n-7$. Hence, it is empty for $n<7$.
\end{Theorem}
\begin{Remark}
This version is a consequence of regularity and compactness results due to R. Schoen and L. Simon in \cite{Schoen-Simon} and Corollary 17.8 in \cite{SIMON}. See also  \cite{Wickramasekera}.
\end{Remark}


In very recent works about regularity and compactness of stable CMC integral varifolds C. Bellettini and N. Wickramasekera, in \cite{Bellettini-Wickramasekera}, \cite{Bellettini-Wickramasekera1} and \cite{Bellettini-Wickramasekera2}, have extended the previous theorem to the class of all stable CMC integral varifolds in the following way: 

\begin{Theorem}[General compactness theorem]\label{Bellettini-Wickramasekera}
Let $\left\{S_{i}\right\}$ be a sequence of stable CMC hypersurfaces in Riemannian manifold $\bar{M}^{n+1}$, without boundary, satisfying  
\begin{equation*}
\limsup\{\mu_{S_i}(\bar{M})+|H_i|\}<\infty,
\end{equation*}
where $\mu_{S_i}$ is the (Riemannian) measure associated to $S_i$ and $H_i$ is the mean curvature of $S_i$. Suppose that there exists a sequence $\{p_i\}$ so that $p_i\in S_i$ and $p_i\to p_{\infty}\in \bar{M}.$ Then there exist a real number $H_{\infty}$, a closed set $\operatorname{sing}S_{\infty}$, an integral varifold $S_{\infty}$ and subsequence  $\left\{S_{i_j}\right\}$ and  subsequence of $\left\{H_{i_j}\right\}$ such that $S_{i}\rightharpoonup S_{\infty}$, $H_i\to H_{\infty}$ and $p_\infty\in \operatorname{spt} S_\infty$. Moreover $\operatorname{sing} S_{\infty}\setminus\operatorname{sing}_T S_{\infty}$ has Hausdorff dimension at most $n-7$, and $\operatorname{sing}_T S_{\infty}\subset \operatorname{gen-reg} S_{\infty}$, $\operatorname{sing}S_{\infty}$ is locally contained in a smooth submanifold of dimension $n-1$,
and $\operatorname{gen-reg} S_{\infty}$ is a classical CMC immersion with mean curvature $H_{\infty}$.
\end{Theorem}
\begin{Remark}
Both Theorem \ref{Wickramasekera} and Theorem \ref{Bellettini-Wickramasekera}, as stated above, are weak versions of the compactness theorems by Bellettini and Wickramasekera.
\end{Remark}

In the statement of Theorem \ref{Bellettini-Wickramasekera} the sets $\operatorname{sing}_C S_{\infty}$, $\operatorname{reg}_1 S$, $\operatorname{sing}_T S_{\infty}$ and $\operatorname{gen-reg} S_{\infty}$  are defined as follows. Here $B_{\varepsilon}(x)$ denotes an open ball of radius $\varepsilon$ and center $x$ in $\bar{M}.
$ 
\begin{Definition}[Set of classical singularities]
$\operatorname{sing}_C S_{\infty}$ denotes the set of all $x\in\operatorname{spt}S$  such that, for some $\alpha\in(0,1]$, there exists $\varepsilon>0$ so that $\operatorname{spt}S\cap B_{\varepsilon}(x)$ is  the union of three or more embedded $C^{1,\alpha}$ hypersurface with boundary meeting pairwise only along their common $C^{1,\alpha}$ boundary $\gamma$ containing $x$ and such that at least one pair of the hypersurfaces with boundary meet transversely everywhere along $\gamma.$ 
\end{Definition}
\begin{Definition}[$C^1$ regular set]
$\operatorname{reg}_1 S$ denotes the set of all $x\in\operatorname{spt}S$  such that there exists $\varepsilon>0$ so that $\operatorname{spt}S\cap B_{\varepsilon}(x)$ is an embedded hypersurface of $B_{\varepsilon}(x)$ of class $C^1.$ 
\end{Definition}
\begin{Definition}[Set of touching singularities]
$\operatorname{sing}_T S_{\infty}$ denotes the set of all $ x\in\operatorname{sing} S_{\infty}\setminus \operatorname{reg}_1S_{\infty}$ so that $x\notin \operatorname{sing}_C S_{\infty}$ and there exists $\varepsilon>0$ such that $\operatorname{spt}S_\infty\cap B_{\varepsilon}(x)$ is the union of two embedded $C^{1,\alpha}$ hypersurfaces of $B_{\varepsilon}(x)$ which touch at $x$.
\end{Definition}
\begin{Definition}[Generalized regular set]
 $\operatorname{gen-reg} S_{\infty}$ denotes the set of all points $x\in\operatorname{spt}S_{\infty}$ so that either $x\in\operatorname{reg} S_{\infty}$ or $x\in\operatorname{sing}_T S_{\infty}$ and there exists $\varepsilon>0$ such that $\operatorname{spt}S_\infty\cap B_{\varepsilon}(x)$ is the union of two embedded $C^{\infty}$ hypersurfaces of $B_{\varepsilon}(x).$ 
\end{Definition}

\section{Translating graphs}\label{TM}

In this section we present some preliminaries on translating and minimal graphs. We will first study the behaviour of the translating graphs in the product space $M\times\R$ and in the end we prove our main theorem for  translating and minimal graphs. Here $M$ is a complete Riemannian manifold  whose metric is denoted by $\sigma$. 
Following \cite{Lira} we say that a hypersurface $\Sigma$ in $M\times\mathbb{R}$ is a translating soliton with respect to the parallel vector field $X=\partial_t$  (with translation speed $c\in\mathbb{R}$) if 
\begin{equation*}
\textbf{H}=c\,X^{\perp},
\end{equation*}
where $\textbf{H}$ is the mean curvature vector field of $\Sigma$ and $\perp$ indicates the projection onto the normal bundle of $\Sigma$. In particular, if $N$ is a normal vector field along $\Sigma$, then we have
\begin{equation}\label{TS}
H=c\langle X, N\rangle,
\end{equation}
where $\langle \cdot, \cdot\rangle$ denotes the Riemannian product metric in $M\times\mathbb{R}$.
 
Although translating solitons seem to define a completely new class of hypersurfaces, the next result due to T. Ilmanen \cite{Ilmanen} tells that they are actually minimal hypersurfaces with respect to the so-called Ilmanen's metric
\begin{equation}
g_c = e^{\frac{2c}{m}t} (\sigma+{\rm d}t^2)
\end{equation}
which is a metric conformal to the Riemannian product metric $g_0 = \sigma+{\rm d}t^2$ in $M\times \mathbb{R}$, see also \cite{Lira}. 

Hence (\ref{TS}) can be interpreted as the fact that the mean curvature of $\Sigma$ with respect to $g_c$ vanishes, i.e.
\begin{equation}
\label{HIH}
e^{\frac{c}{m}t}\widetilde H = H- c\langle X, N\rangle =0.
\end{equation}
This is the Euler-Lagrange equation (for normal variations) of the volume functional 
\[
\mathcal{A}_c[\Sigma]= \int_\Sigma e^{c\eta} {\rm d}\mu_{\Sigma}
\]
where $\eta = t|_\Sigma$ and $e^{c\eta}\,{\rm d}\mu_{\Sigma}$ is the volume element in $\Sigma$ induced by $g_c.$ 
\begin{Remark}
In geometric measure theory, $\mathcal{A}_c[\Sigma]$ is usually called the mass of the varifold associate to $\Sigma.$ Indeed,  the standard notation to $\mathcal{A}_c[\Sigma]$ is $M_c(\Sigma),$ where we put the index $c$ to indicate the dependence of the metric $g_c.$ However, we will use the notation $\mathcal{A}_c[\Sigma]$ to denote the integral above.
\end{Remark}

\begin{Lemma}[T. Ilmanen]
Translating solitons with translation speed $c\in\mathbb{R}$ are minimal hypersurfaces in the product $M\times\R$ with  respect to the Ilmanen's metric $g_c=e^{\frac{2c}{m}t}
(\sigma+{\rm d}t^2)$.
\end{Lemma}

With this interpretation, it is natural to consider the second variation of the volume and the corresponding Jacobi operator
\begin{equation*}
L_c[v]=\Delta v+c\langle X, \nabla v\rangle+(|A|^2+\overline{\operatorname{Ric}}(N,N))v, \quad v\in C^2(\Sigma),
\end{equation*}
where $|A|$ is the norm of the second fundamental form and $\overline{\operatorname{Ric}}$ is the Ricci curvature of $M\times\R$, both calculated with respect to the Riemannian product metric. We refer to \cite{ALR}, \cite{Lira} and \cite{Zhou} for further details. Taking derivatives on both sides of (\ref{TS}) we obtain
\begin{equation}
\label{nablaHkk}
\nabla H=-c\,AT,
\end{equation}
where $T=X^\top$ is the tangential component of $X$. Since $X$ is parallel in $M\times\R$ with respect to the product metric $g_0$, we also have
\begin{equation}
\nabla_VT=\frac{H}{c}AV
\end{equation}
for any $V\in \Gamma(T\Sigma)$. Hence, using Codazzi equation for any $U,V\in\Gamma(T\Sigma)$ we have 
\begin{eqnarray*}
\langle \nabla_U \nabla H, V\rangle & = &  -c\langle (\nabla_U A)T,
V\rangle  -
H\langle AU, AV\rangle\\
\nonumber {} & = & -c\langle (\nabla_T A)U,
V\rangle +c\langle\bar{R}(U,T)N,V\rangle
 -H\langle AU, AV\rangle.
\end{eqnarray*}
Therefore
\begin{equation*}
\Delta H=-c\langle\nabla H,T\rangle+c\,{\rm Ric}_{\bar M}(T,N)-|A|^2H.
\end{equation*}
Since
\[
0 = \overline{\operatorname{Ric}} (X, N) = \overline{\operatorname{Ric}} (T, N) + \langle X, N\rangle \overline{\operatorname{Ric}} (N,N)=\overline{\operatorname{Ric}} (T, N) +\frac{1}{c} H \overline{\operatorname{Ric}}(N,N)
\]
we conclude that 
\begin{equation}
L_c[H] = \Delta H + c\langle X, \nabla H\rangle +(|A|^2+{\rm Ric}_{\bar M}(N,N))H =0.
\end{equation}
Therefore, the function $h:=\langle X,N\rangle$ satisfies  $L_c [h]=0$. 

\vspace{3mm}

Outside the points where a translating soliton is vertical, it can be described locally in non-parametric terms as a graph
\[
\Sigma = \{(x, u(x)): x\in \Omega\}
\]
of a smooth function  $u$ defined in a domain $\Omega\subset M$ with regular boundary (possibly empty.)  In this case, we denote $\Sigma = {\rm Graph}[u]$ and we refer to those solitons as \emph{translating graphs}. From \eqref{TS} we can check that $u$ satisfies the following partial differential equation
\begin{equation}\label{soliton}
\Div\left(\frac{\nabla u}{W}\right)=\frac{1}{W},
\end{equation}
where $W \coloneqq \sqrt[]{1+|\nabla u|^2 }$, and the gradient and divergence operators are taken with respect to the Riemannian metric $\sigma$ of $M$. In this case, $\Sigma$ can be oriented by the normal vector field
\[
N=\frac{1}{W} (X - \nabla u)
\]
with $\nabla u$ translated from $x\in \Omega$ to the point $(x, u(x))\in \Sigma$. Proceeding as in \cite{Zhou} we prove the following result. 
\begin{Lemma}[Shahriyari-Zhou]
All translating graphs are stable in $M\times\R$ endowed with Ilmanen's metric $g_c$.
\end{Lemma}
\begin{Remark}
If $\Sigma$ is a {\em graphical} translator  and $N$ is  the (upward-pointing) unit-normal vector field to $\Sigma$, 
then $\langle \partial_t , N \rangle$ is a positive  Jacobi field.  Consequently, $\Sigma$ is a stable $g_c$-minimal surface. Therefore a sequence  of translating graphs will converge, subsequentially, to a  translator. 
Moreover the vertical translates of $\Sigma$ are also  $g_c$-minimal and foliate a cylinder $\Omega\times\R$, where
$\Omega$ is the region over which $\Sigma$ is a graph. As a consequence, $\Sigma$ is a $g_c$-area minimizing
 surface in $\Omega\times\R$.
\end{Remark}

The proof of the last assertion in the previous remark is quite simple, as we will see in the next lemma. Recall that a complete hypersurface is 
called area-minimizing if any 
compact piece is area-minimizing among all the hypersurfaces with the same boundary.
\begin{Lemma}\label{Homology}
Let $\Omega\subset M$ be a bounded domain and $\Sigma$ a translating graph over $\bar{\Omega}$. For any hypersurface $\Sigma'$ in $\bar{\Omega}\times\mathbb{R}$ with $\partial\Sigma=\partial \Sigma',$ we have
\begin{equation*}
\mathcal{A}_c[\Sigma] =\int_{\Sigma}e^{ct} \, {\rm d}\mu_\Sigma\leq \int_{\Sigma'} e^{ct} \,{\rm d}\mu_{\Sigma'}= \mathcal{A}_c[\Sigma'],
\end{equation*}
and the equality holds  if, and only if, $\Sigma=\Sigma'$. 
\end{Lemma}

\begin{proof} Let $\Psi_\tau (x,t) =\Psi ((x,t), \tau) = (x, t+\tau)$ the flow generated by $X=\partial_t$ in $M\times \mathbb{R}$.  Consider the one-parameter family of translated copies of $\Sigma$ given by
\[
\Sigma_\tau = \Psi_\tau (\Sigma)
\]
and let $N_\tau$ be the vector field in $\bar\Omega \times\mathbb{R}$ given by 
\[
N_\tau (\Psi_\tau(x,t)) = \Psi_{\tau *}(x,t) \cdot N(x, t),
\]
where $N$ is a unit normal vector field along $\Sigma$. It is obvious that $N_\tau$ is a normal vector field along $\Sigma_\tau$. Consider the vector field in $\bar\Omega \times \mathbb{R}$ defined by
\begin{equation}
Y = e^{ct} N_\tau. 
\end{equation}
Let $\{{\sf e}_i\}_{i=1}^m$ be a local orthonormal frame tangent to $\Sigma_\tau$. Therefore, denoting the Riemannian connection and the divergence in $(M\times\mathbb{R}, g_0)$ by $\bar\nabla$ and $ \operatorname{div}$, we obtain 
\begin{eqnarray*}
& & \operatorname{div} Y|_{\Sigma_\tau} = ce^{ct} \langle \bar\nabla t, N_\tau\rangle + e^{ct} {\rm div} N_\tau \\
& & \,\,= ce^{ct} \langle X, N_\tau\rangle +e^{ct}\sum_{i=1}^m \langle \bar\nabla_{{\sf e}_i} N_\tau, {\sf e}_i\rangle + e^{ct}\langle \bar\nabla_{N_\tau} N_\tau, N_\tau\rangle\\
& &\,\, = ce^{ct} \langle X, N\rangle - e^{ct} H_\tau, 
\end{eqnarray*}
where we used the fact that $|N_\tau|^2=1$. Here $H_\tau(\Psi_\tau(x,t)) = H(x,t)$ is the mean curvature of $\Sigma$. Using (\ref{TS}), we conclude that
\begin{equation}
\operatorname{div} Y = e^{ct}(H-c\langle X, N\rangle) =0
\end{equation}
in $\bar\Omega \times\mathbb{R}$. First, suppose that  $\Sigma'\subset \bar{\Omega}\times\mathbb{R}$ lies in one side of $\Sigma$ with $\partial\Sigma=\partial \Sigma'$ and denote by $U$ the domain bounded by $\Sigma\cup \Sigma'$. We have
\begin{eqnarray}
\nonumber  0 &=&\int_U\operatorname{div}Y=\int_{\Sigma} \langle Y,N\rangle\, {\rm d}\mu_{\Sigma}-\int_{\Sigma'}\langle Y, N'\rangle\, {\rm d}\mu_{\Sigma'}\\
\nonumber &  = &\int_{\Sigma}e^{ct} \,{\rm d}\mu_{\Sigma}-\int_{\Sigma'} e^{ct} \langle N, N'\rangle\, {\rm d}\mu_{\Sigma'}\\
\nonumber &  \geq & \int_{\Sigma}e^{ct} \, {\rm d}\mu_{\Sigma}-\int_{\Sigma'}e^{ct}\, {\rm d}\mu_{\Sigma'},
\end{eqnarray}
where $N'$ and ${\rm d}\mu_{\Sigma'}$ define the orientation and the volume element in $\Sigma'$, respectively. This completes the proof in this particular case.  The general case follows by breaking up the hypersurface $\Sigma'$ into regions that lie on one side of $\Sigma$ and applying the previous argument to get the inequality in each one of these regions. 
\end{proof}
\begin{Remark}
Similar result was proved by Y. L. Xin \cite{XIN16} for translating graphs in $\R^{n+1}.$ 
\end{Remark}
\begin{Remark}
Notice that by Lemma \ref{Homology} translating graphs are, in fact, area-minimizing with respect to $g_c$. 
\end{Remark}

Next we provide a proof for the first equality in (\ref{HIH}), that is, the relation between the mean curvatures with respect to $g_c$ and $g_0$. 
The Riemannian connections $\widetilde\nabla$ and $\nabla$ for the metrics $g_c$ and $g_0$, respectively, are related by
\[
\widetilde\nabla_V W = \nabla_V W +\frac{c}{m} \big(\langle V, \partial_t\rangle W + \langle W, \partial_t\rangle V - \langle V, W\rangle \partial_t\big).
\]
If $N$ is a unit normal vector field along $\Sigma$ with respect to $g_0$, here it is not necessary to suppose that $\Sigma$ is a graph. Then the normal vector field with respect to $g_c$ is given by $\widetilde N = e^{-\frac{c}{m}t}N$. Therefore, the second fundamental forms $\widetilde{II}$ and $II$ of $\Sigma$ with respect to $g_c$ and $g_0$ are related by
\[
\widetilde {II} = e^{\frac{c}{m}t} (II - \frac{c}{m}\langle \partial_t, N\rangle g_0|_\Sigma)
\]
Taking traces with respect to $g_c|_\Sigma$ one gets
\begin{equation}
\label{HHI}
\widetilde H = e^{-\frac{c}{m}t}(H -c\langle \partial_t, N\rangle)
\end{equation}
as we have stated above.


As an application of (\ref{HHI}), we have the next lemma.
\begin{Lemma}\label{Product Lemma}
Suppose that $\Lambda$ is a hypersurface in $M$. Then the mean curvature $\widetilde H_{\Lambda\times\R}$ of $\Lambda\times\R$ in $(M\times\R, g_c)$ is given by
\begin{equation}
\label{HHI2}
\widetilde H_{\Lambda\times\R}(x,t)=e^{-\frac{c}{m}t}H_{\Lambda}(x).
\end{equation}
for all $(x,t)\in \Lambda\times \mathbb{R}$. Here $H_\Lambda$ is the mean curvature of $\Lambda$ in $(M, \sigma)$.
\end{Lemma}
\begin{proof} First, observe that 
\[
H_{\Lambda\times\mathbb{R}}(x,t) = H_\Lambda(x)
\]
for all $(x,t)\in \Lambda\times \mathbb{R}$, where both mean curvatures are calculated with respect to $g_0$. Moreover, $\langle \partial_t, N\rangle=0$ along $\Lambda\times\mathbb{R}$ since $N$ is merely the horizontal lift of the unit normal vector field along $\Lambda$ in $M$. Hence, (\ref{HHI}) yields (\ref{HHI2}).
\end{proof}

Before proving the main theorem of this section, we recall what we mean by a complete graph.
\begin{Definition}
Let $\Omega\subset M$ be a domain, not necessarily regular, and let $\Lambda \subset \partial \Omega$ be a smooth open set. We say that a smooth function $u \colon\Omega\to\R$ is complete as we approach $\Lambda$, if $$\lim_{x \to x_0}u(x)=\pm\infty, \mbox{ for any $x_0 \in \Lambda.$}$$
\end{Definition} 
\begin{Theorem} 
\label{main}
Let $M$ be a complete Riemannian manifold and $\Omega\subset M$ be a domain {\rm(}not necessarily regular{\rm)}. Let $\Lambda \subset \partial \Omega$ be a smooth open set and $\Sigma$ a translating or minimal graph of a smooth function $u \colon\Omega\to\R$ that is complete as we approach $\Lambda$. Then $H_{\Lambda}=0.$
\end{Theorem}
\begin{proof}
Fix $x_0\in\Lambda$ and take a sequence $\{x_i\}$ in $\Omega$ with $x_i\to x_0.$ Assume that $u(x_i)\to\infty$ and define the sequence of hypersurfaces  $\{\Sigma_i:={\rm Graph}[u-u(x_i)]\}$ in $M\times\R$ endowed with the Ilmanen's metric $g_c$. 

Fix a closed ball $B$ around $(x_0,0)$ in $M\times\R$ so that $B$ does not intersect $\partial \Lambda\times\R.$ Notice that each $\Sigma_i$ intersects $B$ for $i$ sufficiently large. So, up to subsequence, we can suppose that each $\Sigma_i$ intersects $B$. For each $i$ let $S_i$ be the connected component of $\Sigma_i\cap B$ so that $(x_i,0)\in S_i.$ By Lemma \ref{Homology}
\begin{equation*}
\mathcal{A}_c[S_i]\leq\mathcal{A}_c[\partial B],
\end{equation*}
for all $i.$ Hence, by Theorem \ref{Wickramasekera}, up to a subsequence, we may assume $S_i\to S_{\infty}$ in $\operatorname{int}B\setminus \operatorname{Sing}(S_{\infty})$ and $S_i\rightharpoonup S_{\infty}$ in $\operatorname{int}B.$ Note also that $(x_0,0)\in(\Lambda\times\R)\cap \operatorname{spt}S_{\infty}$. 
Then we have two possibilities for $(x_0,0)$: either it is a regular point or not.

Suppose first that $(x_0,0)$ is a regular point of $S_{\infty}.$ We claim that a small neighbourhood of $(x_0,0)$ in $S_{\infty}$ lies on the cylinder $\Lambda\times\R.$ In fact, if this does not hold for any small neighbourhood of $(x_0,0)$, then we could choose a small domain $S$ in $S_{\infty}$ near $(x_0,0)$ and away from the  singular set, such that $S$ lies in $\Omega\times\R.$ Therefore, we could take a small compact cylinder $C$ through $S$ in $M\times\R$ such that it does not touch $\partial\Omega\times\mathbb{R}.$ Since $S_i\to S_{\infty}$ in $\operatorname{int}B\setminus\operatorname{Sing}(S_{\infty}),$ it follows that, for $i$ sufficiently large, $S_i$ must intersect $C.$ However the assumption $u(x_i)\to\infty$ implies that, for sufficiently large $i$, $S_i\cap C=\varnothing$, which is a contradiction. Therefore a neighbourhood of $(x_0,0)$ in the minimal surface $S_\infty$ lies on $\Lambda\times\mathbb{R}$ and, in particular, we conclude that $\widetilde H_{\Lambda\times\R}(x_0,0)=0$, that is, the mean curvature of $\Lambda\times \mathbb{R}$ with respect to $g_c$ vanishes at $(x_0,0)=0$.

If $(x_0,0)$ is not a regular point of $S_{\infty}$, take a neighborhood $W$ of $(x_0,0)$ in $S_{\infty}$. As the Hausdorff dimension of $\Sing(\Sigma_{\infty})$ is less than $n-7$, we know that 
$W\setminus\Sing (S_\infty)$ is an open dense subset of $W$. Furthermore, we can apply the previous argument to prove that any connected component of  $W\setminus\Sing (S_\infty)$, which is regular, must lie on $\Lambda\times\mathbb{R}$. Hence, we can take a sequence $\{y_i\}\subset \Lambda\times\R$ such that $y_i\to (x_0,0)$ and $\widetilde H_{\Lambda\times\R}(y_i)=0. $ By continuity $\widetilde H_{\Lambda\times\R}(x_0,0)=0.$ Therefore in both cases, Lemma \ref{Product Lemma} implies that we must have $H_{\Lambda}=0$ on $\Lambda$. 
\end{proof}
Finally, using Theorem \ref{Bellettini-Wickramasekera} and the idea of the proof of Theorem \ref{main}, we can obtain a proof of a result that was obtained by M. Eichmair, J. Metzger for $2 \leq n \leq 7$ \cite[Appendix B]{eichmair-metzger}.
\begin{Remark} \label{re:CMC}
Let $M$ be a complete Riemannian manifold and $\Omega\subset M$ be a domain whose boundary is not necessarily regular. Let $\Sigma$ be a graph of a smooth function $u\colon \Omega \to \R$ with constant mean curvature $H_0>0$. Let $\Lambda \subset \partial \Omega$ be a smooth open set  and suppose that $u$ is complete as we approach $\Lambda$.
\begin{enumerate}
\item[{\rm a.}]If $u\to\infty$ on $\Lambda$, then $H_{\Lambda}=H_0$ with respect to the inward normal to $\partial\Omega.$
\item[{\rm b.}]If $u\to-\infty$ on $\Lambda$, then $H_{\Lambda}=-H_0$ with respect to the inward normal to $\partial\Omega$.
\end{enumerate}
\end{Remark}
The idea of the proof (which was already present in \cite{eichmair-metzger}) is the following.  We fix $x_0\in\Lambda$ and take a sequence $\{x_i\}$ in $\Omega$ with $x_i\to x_0.$ Assume first that $u(x_i)\to\infty$ and define the sequence of hypersurfaces  $\{\Sigma_i:={\rm Graph}[u-u(x_i)]\}$ in $M\times\R$. Fix a closed ball $B$ around $(x_0,0)$ in $M\times\R$ so that $B$ does not intersect $\partial \Lambda\times\R$, and suppose that each $\Sigma_i$ intersects $B$. For each $i$ let $S_i$ be the connected component of $\Sigma_i\cap B$ so that $(x_i,0)\in S_i.$ Reasoning as in Lemma \ref{Homology}, we deduce that
\begin{equation*}
\mathcal{A}_0[S_i]\leq\mathcal{A}_0[\partial B]+H_0\operatorname{Vol}(B),
\end{equation*}
for all $i.$ Moreover, it is easy to check that $S_i$ is stable for all $i.$ Hence Theorem \ref{Bellettini-Wickramasekera} implies that, up to a subsequence, we may assume $S_i\rightharpoonup S_{\infty}$ in $\operatorname{int}B$ and $(x_0,0)\in(\Lambda\times\R)\cap \operatorname{spt}S_{\infty}$.\ 

We claim that $\operatorname{spt} S_{\infty}\subset\Lambda\times\R.$ Indeed, suppose that $y_0\notin\Lambda\times\R$ and take any small ball $B'$ around $y_0$ so that it does not intersect $\Lambda\times\R.$ If $\varphi$ is any smooth function with support in $B',$ our definition of $S_i$ give us that 
\[
\int_{B}\varphi{\rm d}\mu_{S_i}=0
\] 
for all sufficiently large $i$. Hence,
\begin{equation*}
\int_{B}\varphi{\rm d}\mu_{S_{\infty}}=\lim_i\int_{B}\varphi{\rm d}\mu_{S_i}=0.
\end{equation*}  
This proves that $y_0\notin\operatorname{spt}S_{\infty}.$ Consequently we must have $\operatorname{spt} S_{\infty}\subset\Lambda\times\R,$ and by regularity of $S_{\infty}$ according to Theorem \ref{Bellettini-Wickramasekera}, we can argue as in Theorem \ref{main} and conclude that $H_{\Lambda}=H_0.$ On the other hand, if $u(x_i)\to-\infty,$  we can argue as above and get $-H_{\Lambda}=H_0,$ where $-H_{\Lambda}$ is the mean curvature with respect to the outward normal to $\Omega$. So $H_{\Lambda}=-H_0$ with respect to the inward normal to $\Omega$.

\subsection{The Euclidean case}

In the Euclidean case $M=\R^n$ endowed with the Euclidean metric $\left<\cdot, \cdot \right>$,
we can obtain a better result. In this particular case, it is very natural to impose on a translators $\Sigma \subset \R^{n+1}$ the condition that
\begin{equation}  \label{eq:entropy}  
\sup_{x \in \R^{n+1}, \; r>0 } \left( \frac{\Area ( \Sigma \cap B(x,r) )}{r^n} \right) < \infty. 
\end{equation}
In particular, if $\Sigma$ arise as a blow up of some mean curvature flow, then it has to satisfy \eqref{eq:entropy}
by \cite[Corollary 2.13]{colding-minicozzi}.

The theorem is then:

\begin{Theorem}\label{eucl-thm}
Suppose $u\colon \Omega \subset \R^n \rightarrow \R$ is  a smooth  function whose graph is a complete translator
satisfying \eqref{eq:entropy}. Then the following holds.
\begin{enumerate}[(i)]
\item  If $n<8$, then $\partial\Omega$ is a smooth minimal hypersurface.
\item  For general $n$, $\partial\Omega$ is a smooth minimal hypersurface
except for a closed singular set of Hausdorff dimension at most $n-8.$
\end{enumerate}
\end{Theorem}

\begin{proof}
Fix $x_0\in\partial \Omega$ and take a sequence $\{x_i\}$ in $\Omega$ with $x_i\to x_0.$ As before,
we define the sequence of hypersurfaces  $\{\Sigma_i:={\rm Graph}[u-u(x_i)]\}$ in $\R^{n+1}$ 
endowed with the Ilmanen's metric $g_c$. 
The idea of the proof is that the result of Schoen and Simon \cite{Schoen-Simon}
shows that a weak limit $V$ of smooth $n$-dimensional stable minimal hypersurfaces, 
is smooth except for an $(n-7)$-dimensional singular set, assuming we
have local area bounds.  These local area bounds are provided by the assumption \eqref{eq:entropy}.

Furthermore, any point of $V$ where the tangent cone is a plane (possibly with
multiplicity) is a regular point.
Now we can apply the Schoen-Simon theorem to $V = (\partial \Omega)\times \R,$ which is the limit of
our sequence $\Sigma_i$.
At first glance, it seems that we get a singular set of dimension at most $n-7$, not $n-8.$
But since $V$ is translation-invariant in the vertical direction, in fact one can get one dimension
better.
One way to see that is the following.

Since $(\partial \Omega) \times \R$ is stable (and has a small singular set) for the Ilmanen metric,
it follows that for any sufficiently small ball $B$ in $\R^n$,  $(\partial \Omega) \cap B$ 
is stable for the Euclidean metric.
Hence the Schoen-Simon theorem implies that the singular set of $(\partial \Omega)\cap B$
has dimension at most $n-8.$
\end{proof}

\begin{Remark}
Under additional hypotheses on the geometry of the translator, it is possible to prove
that $\partial \Omega$ is not only minimal, but totally geodesic $($see \cite{Hoffman}.$)$
In the case of $\R^n$, this means that $\partial \Omega$ consists of  $($disjoint$)$ affine hyperplanes.
\end{Remark}

\section{Jenkins-Serrin Translating Solitons}\label{Jenkins-Serrin Translating Soliton}

In this section we study the existence of Jenkins-Serrin solutions for the translating soliton equation on Riemannian surfaces $M$. To be more precise, we will prove an existence theorem for  type 1 Jenkins-Serrin solutions for the translating soliton equation. 
Let us recall that $u$ is a solution of the translating soliton equation if
\begin{equation*}
\Div\left(\frac{\nabla u}{W}\right)=\frac{1}{W},
\end{equation*}
where $W = \sqrt[]{1+|\nabla u|^2 }$, and the gradient and divergence are taken with respect to the Riemannian metric $\sigma$ of $M$.  
The principal concepts we are going to need are the following:

\begin{Definition}[Nitsche curve]
Let $\Omega$ be a domain in $M$ and $\Gamma\subset M\times\R$ be a Jordan curve. We say that $\Gamma$ is a Nitsche curve, if it admits a parametrization $\Gamma(t) = \{(\alpha(t),\beta(t)) \colon t\in\mathbb{S}^1\}$ such that $\alpha(t)$ is a monotone parametrization of $\partial\Omega$.
 This means that $\alpha \colon \mathbb{S}^1\to\partial\Omega$ is continuous and monotone, and 
 there exist closed disjoint intervals $J_1,\ldots,J_v$ such that $\alpha|_{J_i}$ is constant for all $i$ and $\alpha|_{\mathbb{S}^1\setminus\cup J_i}$ is one-to-one and smooth.
\end{Definition}

\begin{Definition}[Admissible domain]
Let $\Omega$ be a connected domain in $M$. We say that $\Omega$ is an admissible domain if it is geodesically convex and \(\partial\Omega\) is a union of geodesic arcs $A_1,\ldots,A_s,B_1\ldots,B_r$, convex arcs $C_1,\ldots,C_t$, the end points of these arcs and that no two arcs $A_i$ and no two arcs $B_i$ have a common endpoint.
\end{Definition}

\begin{Definition}[Admissible polygon]
Let $\Omega$ be an admissible domain. We say that $\mathcal{P}$ is an admissible polygon if $\mathcal{P}\subset\Omega$ and the vertices of $\mathcal{P}$ are chosen among the vertices of $\Omega.$
\end{Definition}

Let $\Gamma$ be a Nitsche curve over the boundary $\partial\Omega$ of an admissible domain $\Omega$. By a translating soliton with boundary $\Gamma$ we mean a translating soliton in $\Omega\times\R$ that is a graph over $\Omega.$ Using classical results about the solvability of the Plateau problem, we can prove that any Nitsche curve over an admissible domain admits a unique translating soliton with it as the boundary.

\begin{Theorem}[Local existence]\label{Existence}
Let $\Omega$ be an admissible domain in $M$ and $\Gamma$ a Nitsche curve over $\partial\Omega$. Then there exists a unique translating soliton with boundary $\Gamma$. 
\end{Theorem}
\begin{proof}
The proof is similar to that of \cite{Pinheiro} and therefore we skip some details. First we note that $\Omega\times\R$ is homogeneous in the sense of Meeks and Yau \cite{Meeks-Yau-1, Meeks-Yau-2}, and Morrey \cite{Morrey}. Since the boundary $\partial\Omega$ consists of geodesic and convex arcs, the boundary $\partial(\Omega\times\R)$ is mean convex. By Lemma \ref{Product Lemma} it remains mean convex also when we change the metric to Ilmanen's metric $g_c$. Therefore there exists an embedded minimal (w.r.t. $g_c$) disk $\Sigma \subset \Omega\times\R$ with boundary $\Gamma$. It remains to prove that $\operatorname{int}(\Sigma)$ is a graph over $\Omega$. 

First we show that, for all $p\in\operatorname{int}(\Sigma)$, $T_p\Sigma$ is not a vertical plane. On the contrary, suppose that there exists a point $p\in\operatorname{int}(\Sigma)$ such that $p\in M\times\{c\}$ for some $c\in\R$ and that the tangent plane $\Pi$ to $\Sigma$ is vertical in $\Omega\times\R.$ Take a basis $\{\partial_t,\ v\}$ tangent to $\Pi$ at $p$, where $\partial_t$ is tangent to $\Sigma$ and $v$ is tangent to $M\times\{c\}$ with $||v||=1.$ Let $\gamma$ be the unique geodesic in $M\times\{c\}$ such that $\gamma(0)=p$ and $\gamma'(0)=v$. Note that $\gamma$ intersects $\partial(\Omega\times\R)$ exactly in two points. 
 
Now $\gamma\times\R$ is a totally geodesic surface, in particular minimal, in $\Omega\times\R$ and by Lemma \ref{Product Lemma} it is minimal also with respect to the metric $g_c$. Moreover, we have $T_p(\gamma\times\R)=\Pi$ and therefore, near $p$, $I\coloneqq \Sigma\cap(\gamma\times\R)$ contains at least two curves that intersect transversally at $p.$ If there exists a closed curve $\alpha$ in $I\setminus\partial\Sigma$, then $\alpha$ is the boundary of a minimal disk $D$ in $\Sigma.$ Thus we could choose a geodesic curve $\beta$ in $D$ so that the totally geodesic surface $\beta\times\R$ touches $D$ at an interior point. But this is impossible by the maximum principle. 

Since $I$ does not contain a closed curve, each of the branches leaving $p$ must go to $\partial\Sigma$. Moreover, $\gamma$ intersects $\partial\Omega$ at two points so at least two of these branches must go to the same point or vertical segment of $\partial\Sigma$. However, this yields again closed curve that bounds a minimal surfaces and we get a contradiction with the maximum principle. Therefore $T_p\Sigma$ is not a vertical plane,
 


Finally, the same argument as in \cite{Pinheiro} shows that each vertical line in $\Omega\times\R$ intersects $\Sigma$ exactly at one point and therefore $\Sigma$ is a graph over the interior of $\Omega$.

\end{proof}



 
 To prove our main theorem, we will need the following maximum principle.

\begin{Proposition}[Maximum principle]\label{Max. Principle}
Let $\Omega \subset M$ be an admissible domain. Suppose that $u_1$ and $u_2$ satisfy
\[
\Div\left(\frac{\nabla u_{1}}{\sqrt[]{1+|\nabla u_1|^2}}\right)\geq\Div\left(\frac{\nabla u_{2}}{\sqrt[]{1+|\nabla u_2|^2}}\right),
\] 
and \(\liminf(u_2-u_1)\geq0\) for any approach of $\partial\Omega$, with possible exception of finite numbers of points $\{q_1,\ldots,q_r\} \eqqcolon E \subset\partial\Omega.$ Then $u_2\geq u_1$ on $\partial\Omega\setminus E$ with strict inequality unless $u_2=u_1.$
\end{Proposition}
\begin{proof}
The proof follows similar arguments as in \cite{Spruck-Infinite}. Let $K$ and $\varepsilon$ be positive constants, with $K$ large enough and $\varepsilon$ small enough. Define a function
\[
\varphi\coloneqq
\begin{cases}
K-\varepsilon, &\text{if } \  u_1-u_2\geq K; \\
u_1-u_2 -\varepsilon, &\text{if } \ \varepsilon<u_1-u_2\leq K;  \\
0,  &\text{if } \ u_1-u_2 \leq\varepsilon.
\end{cases}
\]

Notice that $\varphi$ is Lipschitz with $0\leq\varphi\leq K$, and $\nabla\varphi=\nabla u_1-\nabla u_2$ in the set $\{\varepsilon<u_1-u_2<K\} $ and $\nabla\varphi=0$ almost everywhere in the complement of $\{\varepsilon<u_1-u_2<K\} $.  Around any point $q_i\in E$, consider an open geodesic disk $B_\varepsilon(q_i)$ of radius $\varepsilon$ and center $q_i$. Let $\Omega_\varepsilon:=\Omega\setminus\cup B_\varepsilon(q_i)$, and suppose that $\partial\Omega_\varepsilon=\tau_\varepsilon\cup\rho_\varepsilon$, where $\rho_\varepsilon=\cup(\partial B_\varepsilon(q_i)\cap\Omega)$ and $\tau_\varepsilon=\partial\Omega_{\varepsilon}\cap\partial\Omega.$ Since \(\liminf(u_2-u_1)\geq0\) in $\partial\Omega\setminus E,$ we have $\varphi\equiv0$ in a neighbourhood of $\tau_\varepsilon.$

Define 
\begin{equation}\label{Equation J}
J:=\int_{\rho_\varepsilon}\varphi\left\{\left\langle\frac{\nabla u_1}{W_1},\nu\right\rangle-\left\langle\frac{\nabla u_2}{W_2},\nu\right\rangle\right\},
\end{equation}
where $\nu$ is the unit outer conormal to $\Omega_\varepsilon$ and $W_i = \sqrt[]{1+|\nabla u_i|^2}.$ From \eqref{Equation J} and $0\leq\varphi\leq K$ we obtain

\begin{equation}\label{Ineq. upper J}
J\leq 2K\sum_{i=1}^{r}||\partial B_\varepsilon(q_i)||,
\end{equation}
where $||\partial B_\varepsilon(q_i)||$ denotes the length of $\partial B_{\varepsilon}(q_i).$ On ther other hand, since $\varphi$ is Lipschitz, we have

\[
\Div\left(\varphi\left\{\frac{\nabla u_1}{W_1}-\frac{\nabla u_2}{W_2}\right\}\right)=\nabla\varphi\left\{\frac{\nabla u_1}{W_1}-\frac{\nabla u_2}{W_2}\right\}+\varphi\left\{\Div\left(\frac{\nabla u_1}{W_1}\right)-\Div\left(\frac{\nabla u_2}{W_2}\right)\right\},
\]
almost everywhere in $\Omega$. By Stokes theorem we get

\begin{align}\label{Ineq. Under J}
J &= \int_{\Omega_\varepsilon}\left\{\left\langle\nabla\varphi,\left(\frac{\nabla u_1}{W_1}-\frac{\nabla u_2}{W_2}\right)\right\rangle+\varphi\left(\Div\left(\frac{\nabla u_1}{W_1}\right)-\Div\left(\frac{\nabla u_2}{W_2}\right)\right)\right\} \nonumber \\ 
&\geq \int_{\Omega_\varepsilon}\left\langle\nabla\varphi,\left(\frac{\nabla u_1}{W_1}-\frac{\nabla u_2}{W_2}\right)\right\rangle.
\end{align}
Now if $N_i:=\frac{\partial_t}{W_i}-\frac{\nabla u_i}{W_i},$ then

\begin{align}\label{Ineq. normal}
\nonumber\left\langle\nabla u_1-\nabla u_2,\left(\frac{\nabla u_1}{W_1}-\frac{\nabla u_2}{W_2}\right)\right\rangle 
&=\left\langle N_1-N_2,W_1N_1- W_2N_2\right\rangle \\
\nonumber &= W_1-(W_1+W_2)\langle N_1,N_2\rangle+W_2\\
&=\frac{1}{2}(W_1+W_2)||N_1-N_2||^2.
\end{align}

From \eqref{Ineq. upper J}, \eqref{Ineq. Under J} and \eqref{Ineq. normal} we get
\[
 2K\sum_{i=1}^{r}||\partial B_\varepsilon(q_i)||\geq\frac{1}{2}\int_{\Omega_\varepsilon\cap\{0<u_1-u_2<K\}}(W_1+W_2)||N_1-N_2||^2\geq0.
\]
In particular, when $\varepsilon\to 0$ we obtain

\[
\int_{\{0<u_1-u_2<K\}}(W_1+W_2)||N_1-N_2||^2=0.
\]
Therefore $N_1=N_2$ in $\{0<u_1-u_2<K\},$ so $\nabla u_1=\nabla u_2$ in $\{0<u_1-u_2<K\}.$ As $K$ was arbitrary we may conclude $\nabla u_1=\nabla u_2$ in the set $\{0<u_1-u_2\}.$ Suppose now that $\{0<u_1-u_2\}$ contains a connected component with non-empty interior. By the previous argument $u_1=u_2+c,$ where $c$ is a positive constant, so by the maximum principle $u_1=u_2+c$ in $\Omega.$ On the other hand, as \(\liminf(u_2-u_1)\geq0\) for any approach of $\partial\Omega\setminus E$, then $c$ is a non-positive constant, which is impossible. This finishes the proof.
\end{proof}


Finally, we can prove the main theorem of this section. 

\begin{Theorem}[Existence of Jenkins-Serrin Solution]
Let $\Omega\subset M$ be an admissible domain with $\{B_i\}=\varnothing.$ Given any continuous data $f_i \colon C_i\to\R$, there exists a Jenkins-Serrin solution $u\colon\Omega\to\R$ for the translating soliton equation with continuous data $u|_{C_i}=f_i$, if for any admissible polygon $\mathcal{P}$ we have
\begin{equation}\label{structure-2}
2\alpha(\mathcal{P})<\ell(\mathcal{P}).
\end{equation}
\end{Theorem}
\begin{proof}
Define a Nitsche curve $\Gamma_n = (\alpha_n,\beta_n)$ by setting $\beta_n=n$ on $\{A_i\}$ and $\beta_n=\min\{f_i,n\}$ on $C_i$ for all $i.$ By Theorem \ref{Existence}, for all $n\in\mathbb{N}$, there exists $u_n \colon \Omega\to\R$ so that ${\rm Graph}[u_n]$ is a translating soliton in $\Omega\times\R$ with boundary $\Gamma_n$. Notice that if $n>m$ we have $u_n\geq u_m$ on $\partial\Omega,$ so $u_n>u_m$ in $\Omega$ by comparison principle. Hence $\{u_n\}$ is a monotone sequence. Taking into account  Pinheiro's results \cite{Pinheiro}, \eqref{structure-2} guarantees that there exists a Jenkins-Serrin solution $v\colon\Omega\to\R$ for the minimal graph equation with continuous data $f_i$. Since \[\Div\left(\frac{v}{\sqrt[]{1+|v|^2}}\right)=0<\frac{1}{\sqrt[]{1+|u_n|^2}}=\Div\left(\frac{u_n}{\sqrt[]{1+|u_n|^2}}\right)\] and $\liminf(v-u_n)\geq0$ on $\partial\Omega\setminus E,$ where $E$ is the set of vertices of $\Omega$, Proposition \ref{Max. Principle} implies $v>u_n$ for all $n.$ Therefore $\lim u_n=u$ exists and $u$ satisfies
\[
\Div\left(\frac{u}{\sqrt[]{1+|\nabla u|^2}}\right)=\frac{1}{\sqrt[]{1+|u|^2}}
\]
in $\Omega.$ Clearly $u|_{C_i}=f_i$, by construction, and $u\to\infty$ as we approach $A_i$ for all $i.$
\end{proof}



\bibliographystyle{amsplain, amsalpha}

\end{document}